\documentclass[11pt,reqno]{amsart}

\usepackage{amsthm,amsmath,amssymb}
\usepackage{graphicx,color}
 \usepackage[colorlinks=true]{hyperref}

\hypersetup{urlcolor=blue, citecolor=red}
\def\MR#1{\href{http://www.ams.org/mathscinet-getitem?mr=#1}{MR#1}}

\def\E{{\mathcal E}}

\def\R{\mathbb{R}}

\def\bal{\begin{aligned}}
\def\eal{\end{aligned}}
\def\vol{\mathrm{Area}}

\newtheorem{lemma}{Lemma}
\newtheorem{corollary}[lemma]{Corollary}
\newtheorem{theorem}{Theorem}

\theoremstyle{definition}

\title[Dirichlet Laplacian eigenvalue of regular Polygons]{On the first Dirichlet Laplacian eigenvalue of regular Polygons}

\author{Carlo Nitsch}
\address{Mathematisches Institut ,
Universit\"at zu K\"oln,
D-50931 Cologne, 
Germany}
\address{and}
\address{Dipartimento di Matematica e Applicazioni, Universit\`a di Napoli ``Federico II'', Complesso di Monte S. Angelo, Via Cintia, 80126 Napoli, Italy}
\email{c.nitsch@unina.it}
\numberwithin{equation}{section}



\begin{document}

\begin{abstract}
The Faber-Krahn inequality in $\R^2$ states that among all open bounded sets of given area the disk minimizes the first Dirichlet Laplacian eigenvalue. It was conjectured in \cite{AF} that for all $N\ge 3$ the first Dirichlet Laplacian eigenvalue of the regular $N$-gon is greater than the one of the regular $(N+1)$-gon of same area. This natural idea is suggested by the fact that the shape becomes more and more ``rounded" as $N$ increases and it is supported by clear numerical evidences. Aiming to settle such a conjecture, in this work we investigate possible ways to estimate the difference between eigenvalues of regular $N$-gons and $(N+1)$-gons.
\end{abstract}

\maketitle

\section{Introduction}

Given an open set $\Omega\subseteq \R^2$ with finite measure, the first Dirichlet Laplacian eigenvalue is the least positive number $\lambda$ such that the boundary value problem
\begin{equation}\label{eigenvalue_prob}
\left\{
\begin{array}{ll}
-\Delta u=\lambda\, u\ &\mbox{in }\Omega,\\
u=0\ &\mbox{on } \partial\Omega
\end{array}\right.
\end{equation}
has non trivial solutions in $H^1_0(\Omega)$. The corresponding solutions are called first Dirichlet Laplacian eigenfunctions. If $\Omega$ is connected then eigenfunctions have constant sign and $\lambda$ is simple (eigenfunctions are unique up to a multiplicative factor).

We recall also that, by classical arguments, $\lambda$ can be characterized as
\begin{equation}\label{ray}
\lambda=\min\left\{\frac{\|D  u\|^2_{L^2(\Omega)}}{\|u\|^2_{L^2(\Omega)}}:\,u\in H_0^1(\Omega),\, u\not\equiv0 \right\}\,
\end{equation}
and a function minimizes \eqref{ray} if and only if it is a first Dirichlet Laplacian eigenfunction.

In this paper we are mainly concerned with so-called isoperimetric inequalities for $\lambda$. In a broad sense, by isoperimetric inequalities we mean a-priori bounds of $\lambda$ when geometric constraints (such as volume, perimeter, circumradius, etc.) on $\Omega$ are prescribed. The most celebrated inequality in such a class is certainly the Faber-Krahn inequality stating that, among all open sets of $\R^2$ of given measure the disk achieves the least possible eigenvalue \cite{Hen,Ke}. 

In the following, when it is needed to better emphasize the domain dependence of $\lambda$, we will use the notation $\lambda(\Omega)$. 

Among the most important properties of the first Dirichlet Laplacian eigenvalue we remind that, by scaling arguments, it holds 
\begin{equation}\label{eq_scaling}
\lambda(\Omega)=t^2\lambda(t\Omega),
\end{equation}
for all real positive $t$. Moreover it is worth mention that, using the variational characterization \eqref{ray}, one can deduce the monotonicity with respect to $\Omega$ in the sense that, whenever $\tilde\Omega\subset\Omega$ are two open sets of finite measure, then
$$\lambda(\Omega)\le\lambda(\tilde\Omega),$$
and the inequality is strict if $\Omega$ is connected.

The last property that we remind is that, if $\Omega$ is connected and symmetric with respect to a rotation or a reflection, the same is true also for the eigenfunctions in view of the simplicity of the eigenvalue.

Making use of \eqref{eq_scaling} the Faber-Krahn inequality reads as

$$\vol(B)\lambda(B)\le\vol(\Omega)\lambda(\Omega)$$

whenever $\Omega$ is open with finite measure, $B$ is a disk, and $\vol(\cdot)$ denotes the measure in $\R^2$.\\

In literature there are many variations on the theme Faber-Krahn, all concerning similar isoperimetric inequalities for the first Dirichlet Laplacian eigenvalue with different or additional constraints. Without claiming to be exhaustive we remind for instance that in \cite{Ma,PW,Po} the author provide upper and lower bounds for convex sets in terms of area and perimeter. The same was done more recently also in \cite{BNT,DP1,DP2,Ni}. Different classical estimates may also include diameter and inradius like in \cite{PS,Pr} while a different approach consists in restricting the class of sets. And indeed from now on we confine our investigation to polygons. Fundamental tone of Dirichlet Laplacian on polygons has been widely investigated for instance in \cite{CK,Fre,GS,So}. Nevertheless  many challenging unsolved questions \cite{AF,Hen} are still unsolved. The most important is due to P\'olya and Szeg\"o \cite{PS} who conjectured that among all $N$-gons of given area the regular one achieves the least possible $\lambda$. The corresponding inequality reads as follows

\begin{equation}\label{fkpoly}\vol(P_N)\ \lambda(P_N)\le\vol(p)\ \lambda(p),\qquad p\in\mathcal{P}_N\end{equation}
where $\mathcal{P}_N$ is the set of all $N$-gons and $P_N\in \mathcal{P}_N$ denotes a regular one.
This conjecture is suggested by the Faber-Krahn inequality in conjunction with the idea that, for a given number of sides, the regular polygon has the most rounded shape. However, in spite of this simple idea, this problem is very challenging and the conjecture has been settled only for $N=3$ and $N=4$ where it is possible to use the Steiner symmetrization \cite{Hen,PS}. 
What is however known (see \cite{Hen}) is that for given $N\ge3$ there exists an $N$-gon which minimizes the product $\vol(p)\ \lambda(p)$ and it is also known that by increasing the number of sides such a minimum decreases, namely: 

$$
\min\{\vol(p)\ \lambda(p):p\in\mathcal{P}_N\}\ge\min\{\vol(p)\ \lambda(p):p\in\mathcal{P}_{N+1}\}.
$$

This of course implies that, the conjectured inequality \eqref{fkpoly} can be true only if it is also true that
$$\vol(P_{N+1})\ \lambda(P_{N+1})\le\vol(P_N)\ \lambda(P_N).$$

Surprisingly enough, to our knowledge, even this inequality is still unproved, as also testified by a recent paper \cite{AF} where, motivated by numerical examples, the authors not only conjectured that along the sequence of regular polygons $\{P_N\}_{N\in\mathbb{N}}$ the product $\vol(P_N)\ \lambda(P_N)$ is decreasing in $N$, but also that the ratio $\displaystyle\frac{\vol(P_{N})\ \lambda(P_N)}{\vol(P_{N+1})\ \lambda(P_{N+1})}$ is decreasing in $N$. \\

From now on by $P_N$ will always denote a regular polygon with $N$ sides and when it is necessary to specify its circumradius $r$ we will use the notation $P_N^r$. \\

Motivated by the lack of analytic estimates which allow to investigate the behavior of $\lambda(P_N)$ for different $N$ in this work we present two possible approach to the problem. The first one is based on the so called dissections which has been used also in \cite{So}; it is a purely geometric technique and gives the following result.
\begin{theorem}\label{teo_geo}
For all $N\ge 3$ and $r>0$ we have 
$\lambda(P_{N+1}^{r})< \lambda(P_N^{r})$.
\end{theorem}
Even if our result is weaker than \eqref{fkpoly}, to our knowledge it is new in the literature. For fixed inradius the reversed inequality can be found in \cite{So}, where actually the author also proves that among $N$-gons of given inradius the regular one achieves the highest eigenvalue. 
 
In the second part of the paper, using a different approach based on the shape derivative (\cite{Hen,HP,NP,SZ}) we then provide a refinement.
\begin{theorem}\label{teo_ref}
For all $N\ge 3$ and $r>0$ we have
\[
\lambda(P_{N+1}^ r)<\lambda(P^ r_{N})\displaystyle\frac{\cos\displaystyle\frac{\pi}{N}}{\cos\displaystyle\frac{\pi}{N+1}}.
\]
\end{theorem}

If by $j_0$ we denote the first zero of $J_0$ ($J_0$ denotes the Bessel function of the first kind and order zero \cite{AS}), then the eigenvalue of the disk of radius $r$ is $\frac{j_0^2}{r^2}$ \cite{Ke}. With such a notation, the sharpest result that we present using the shape derivative is the following.

\begin{theorem}\label{teo_best}
For all $N\ge 3$ and $r>0$ we have
\[\displaystyle \ell^r_{N+1}\rho_{N+1}^{r}\lambda(P_{N+1}^{r})<  \ell_{N}^{r}\rho_{N}^{r}\lambda(P_N^{r})-\frac{2\pi j_0^2}{N(N+1)},\]
 where $\ell_N^r$ and $\rho_N^r$ are the side length of $P_N^r$ and the inradius of $P_N^r$ respectively. 
\end{theorem}

Iterating the previous inequality (summing up over all $K>N$) and taking into account that the eigenvalue of $P_K^r$ goes to $j_0^2/r^2$ as $K\to\infty$, we have the following. 

\begin{corollary}
For all $N\ge3$ we rediscover the Faber-Krahn inequality for regular polygons,
\[\displaystyle \lambda(P_{N})> \frac{\pi j_0^2}{\vol(P_N)}.\]
\end{corollary}
Such a result, although not original, emphasizes that Theorem \ref{teo_best} can be also understood as a refinement of the Faber-Krahn inequality on regular polygons.\\
Finally we observe that Theorem \ref{teo_best} can be rewritten in the following way.

\begin{corollary}
For all $N\ge3$ we have
\begin{equation}\label{ultimate}
\displaystyle \vol(P_{N+1})\lambda(P_{N+1})<  \vol(P_{N})\lambda(P_N)+\frac{\vol(P_{N})\lambda(P_N)-{\pi j_0^2}}{N}.
\end{equation}
\end{corollary}
Unfortunately we are unable to prove or disprove the conjectured inequality \eqref{fkpoly}, since the reminder term $\displaystyle\frac{\vol(P_{N})\lambda(P_N)-{\pi j_0^2}}{N}$ in \eqref{ultimate} is positive. However we make a step forward to its proof, and we provide two different point of view and possibly two useful approaches to the problem.

\section{Proof of Theorem \ref{teo_geo}}

The proof we propose is completely based on the geometric construction of a particular test function.
For simplicity we start by considering a square $P^r_4$ and a regular pentagon $P^r_5$ having the same circumradius $r$.
The square is split into eight polygons (Figure \ref{fig_pentagono}(a)). Four of them, those denoted by $T_i$ ($i=1,...,4$) and represented in grey, are congruent open isosceles triangles. The other four ($Q_i$ with $i=1,...,4$) are congruent open convex quadrilaterals. The four isosceles triangles have one vertex in common which also coincides with the center of the square. The angle at this vertex is equal to ${\pi}/{10}$ which is exactly the difference of the central angle of the square ($\pi/2$) and the pentagon ($2\pi/5$). Here by central angle we mean the angle made at the center of the polygon by any two adjacent vertices of the polygon. \\
Now we can rearrange this eight pieces (and the eight cutting segments), simply by rotating them around the center (see Figure \ref{fig_pentagono}(b)) to form a new irregular open polygon $D$ having the same area as the square. The polygon $D$ is strictly included into a regular pentagon which, by construction, has the same circumradius of the square. This kind of geometric construction is also sometimes called dissection.

For what concerns our purposes, such a dissection can be naturally translated into a bijection $\Phi: P^r_4\to D\subset P^r_5$ with the only important requirement that $\Phi$ must act on each $T_i$ ($i=1,...,4$) and each $Q_i$ ($i=1,...,4$) as a rotation. The map $\Phi$ has then the following interesting properties.
First of all, if $u$ is a first eigenfunction on $P^r_4$, then the function
$$v(x)=\left[\begin{array}{ll}u(\Phi^{-1}(x))& \mbox{ if } x\in D\\
 0 &\mbox{ otherwise}\end{array}\right.$$
is continuous on $P^r_5$. This is true because, in view of the symmetry of $u$ with respect to reflection across the axes of the four sides and across diagonals (the so-called dihedral group $\mathcal {D}_4$), the function $u$ takes the same value on all points of the cutting segment having the same distance to the center of $P_4^r$. Moreover $v$ belongs to $H^1_0(P^r_5)$ since it belongs to $H^1(T_i)$ and to $H^1(Q_i)$ (for $i=1,...,4$).

What more, 
by construction 
$$\int_{P^r_4}u^2=\int_{D}v^2$$
and
$$\int_{P^r_4}|Du|^2=\int_{D}|Dv|^2.$$

Therefore we have that $$\displaystyle\lambda(P^r_5)<\lambda(D)\le\frac{\displaystyle\int_{D}|Dv|^2}{\displaystyle\int_{D}v^2}=\frac{\displaystyle\int_{P^r_4}|Du|^2}{\displaystyle\int_{P^r_4}u^2}=\lambda(P^r_4)$$ which completes the proof of Theorem \ref{teo_geo} for $N=4$. 

However, the very same construction can be applied to any couple of consecutive regular polygons $P^r_N$ and $P^r_{N+1}$ with the same circumradius $r$. In this case we can construct a dissection which splits $P^r_N$ into $2N$ pieces. Again $N$ of them are congruent isosceles triangle sharing the center of $P^r_N$ as one vertex. The angle that these isosceles triangles have in the center is now equal to $\frac{2\pi}{N(N+1)}$, which is exactly the difference between the central angles of $P^r_N$ and $P^r_{N+1}$. Then the rest of the proof can continue exactly as above taking advantage of the symmetry of eigenfunctions with respect to the dihedral group of rotations and reflections $\mathcal{D}_N$.


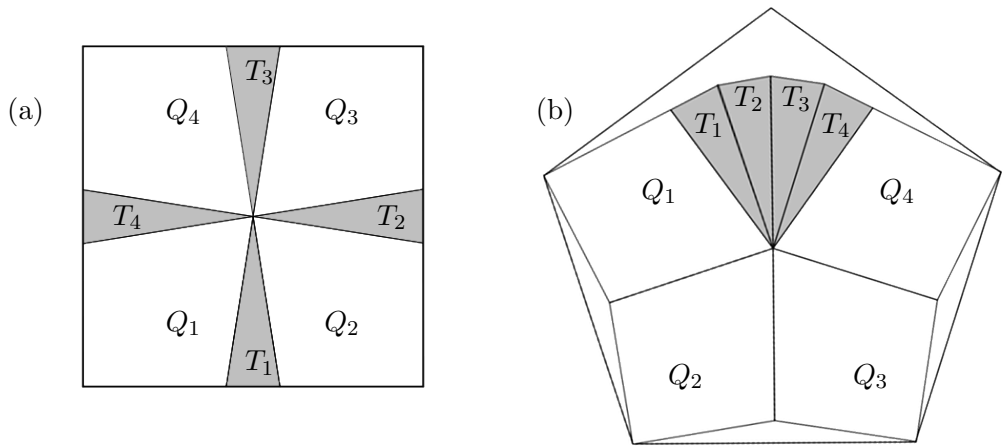
\begin{figure}[!ht]
\centering
\begin{picture}(200,200)
\put(-130,-20){\includegraphics[width=1.00\textwidth]{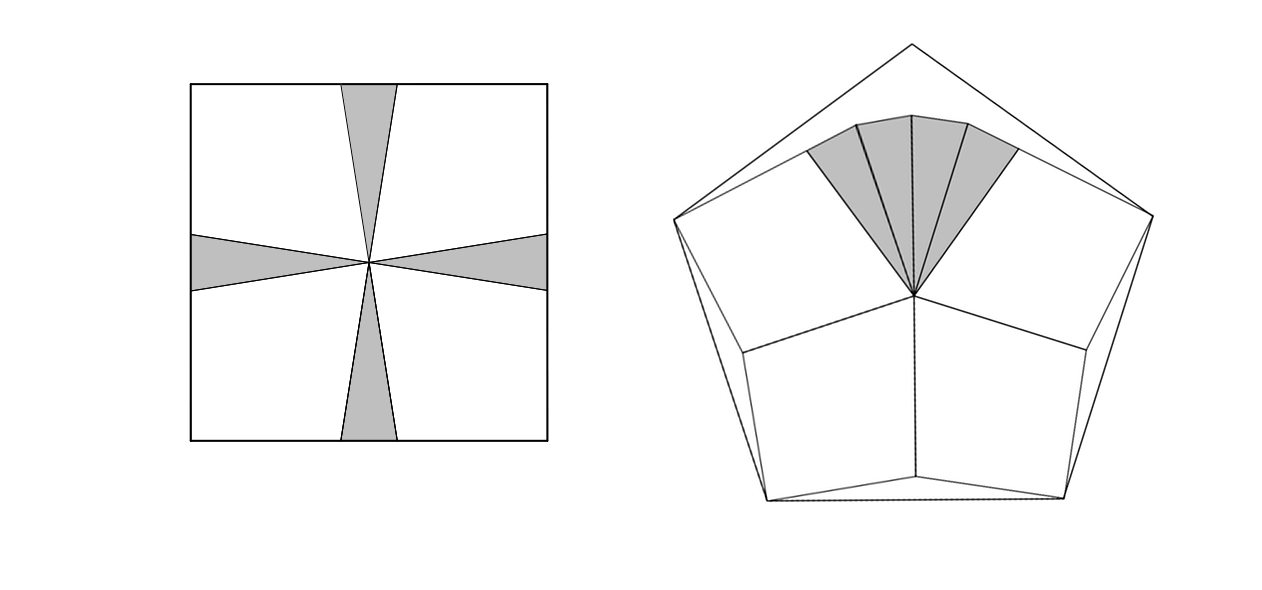}}
\put(-90,135){(a)}
\put(-30,55){$Q_1$}
\put(30,55){$Q_2$}
\put(30,135){$Q_3$}
\put(-30,135){$Q_4$}
\put(0,40){$T_1$}
\put(50,95){$T_2$}
\put(0,150){$T_3$}
\put(-50,95){$T_4$}
\put(110,135){(b)}
\put(150,105){$Q_1$}
\put(160,35){$Q_2$}
\put(230,35){$Q_3$}
\put(240,105){$Q_4$}
\put(170,130){$T_1$}
\put(185,140){$T_2$}
\put(203,140){$T_3$}
\put(218,130){$T_4$}
\end{picture}
        \caption{Rearranging the square into a regular pentagon}
        \label{fig_pentagono}
\end{figure}

\section{Proof of Theorem \ref{teo_best} and Theorem \ref{teo_ref}}

For the reader convenience we split the proof into several lemmata. First we observe that the study of the first Dirichlet Laplacian eigenvalue problem on a regular polygon goes along with the study of a mixed boundary eigenvalue problem on right triangles (see also \cite{AC} for Laplacian eigenvalues with mixed boundary conditions).

\begin{figure}
\def\svgwidth{\textwidth}
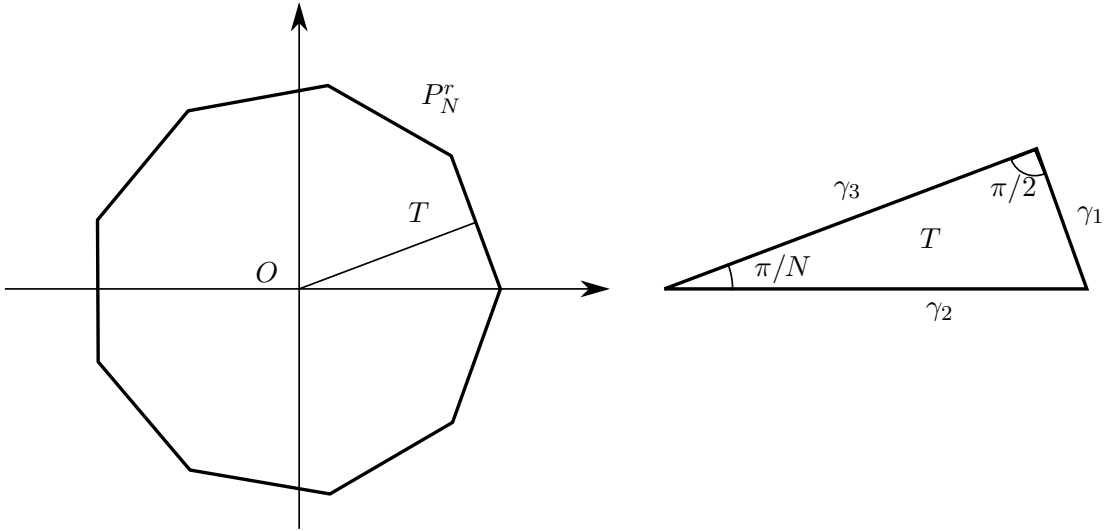\caption{The polygon $P_N^r$ and the right triangle $T$.}\label{PT}
\end{figure}

Let $T=T(\alpha, r)$ be a right open triangle 
with hypotenuse of length $ r$ and one of the acute angle measuring $\alpha$. 
Let us define $\gamma_1$ the catethus opposite to the angle whose measure is $\alpha$, and let $\gamma_2$ and $\gamma_3$ be the hypothenuse and the other cathetus respectively.\\
We define $\mu(T)$ to be the least positive number such that there exists a nontrivial solution to the following problem 
\begin{equation}\label{mixed}
\left\{
\begin{array}{ll}
-\Delta v=\mu(T)\, v\ &\mbox{in }T,\\
v=0\ &\mbox{on } \gamma_1,\\
\frac{\partial v}{\partial \nu}=0\ &\mbox{on } \gamma_2\cup\gamma_3.
\end{array}\right.
\end{equation}
Here $\nu$ is the unit exterior normal to $\partial T$. As for the first eigenfunctions of the Dirichlet eigenvalue problem, using similar classical arguments, it is possible to prove that any solution $v$ has constant sign in $T$.\\
By classical arguments it is easy to see that the first eigenvalue $\mu(T)$ can be also characterized by the variational formulation
\begin{equation}\label{def_mu}
\mu(T)=\min\left\{\frac{\|D v\|^2_{L^2(T)}}{\|v\|^2_{L^2(T)}}:\,v\in H^1(T),\, v\not\equiv0,\, v=0 \;\text{on $\gamma_1$} \right\}\,.
\end{equation}
Moreover $w$ is a minimizer of \eqref{def_mu} if and only if it is a solution to problem \eqref{mixed}. The following lemma holds.

\begin{lemma}
For all $r>0$ and $N\ge 3$, if $\alpha=\pi/N$, then $\lambda(P^r_N)=\mu(T(\alpha, r))$.
\end{lemma}
\begin{proof}
The proof is elementary and based on the symmetry of regular polygons. 
A regular polygon with $N$ sides has 2N different symmetries: $N$ rotational symmetries and $N$ reflection symmetries (forming the so-called dihedral group). \\ In a reference frame (like the one in Figure \ref{PT}) in which the origin is the center of $P_N$ and one of the vertices is on the positive $x$-semiaxis, for $(k=1,...,N)$ the rotations $\mathcal{S}^1_k$ and reflections $\mathcal{S}^2_k$ have the following matrix representation:

\[
\mathcal{S}^1_k=\left(
\begin{array}{cc}
\cos \frac{2\pi k}{N}& -\sin\frac{2\pi k}{N}\\
\sin\frac{2\pi k}{N}&\cos\frac{2\pi k}{N}
\end{array}
\right), \qquad
\mathcal{S}^2_k=\left(
\begin{array}{cc}
\cos \frac{2\pi k}{N}& \sin\frac{2\pi k}{N}\\
\sin\frac{2\pi k}{N}&-\cos\frac{2\pi k}{N}
\end{array}
\right).
\]
\\

Let $P^r_N$ be a regular polygon with center $O$. We can draw, inside $P^r_N$, a triangle which we identify with $T(\pi/N,r)$ by considering the following three vertices (see Figure \ref{PT}): 
\begin{itemize}
\item[(i)] the center $O$, 
\item[(ii)] the midpoint of one of the sides of $P_N$,
\item[(iii)] one of the corners of $P_N$ adjacent to (ii). 
\end{itemize}
We check that by construction the angle corresponding to the vertex in $O$ is equal to $\pi/N$ and moreover the length of the hypothenuse is $r$. 

Now we consider a function $u$ solution to \eqref{eigenvalue_prob} on $P_N^r$ and a function $w$ solution to \eqref{mixed} on $T(\pi/N,r)$.
Since the function $u$ is invariant under the action of the symmetry group of $P^r_N$, we have $\frac{\|D u\|^2_{L^2(T)}}{\|u\|^2_{L^2(T)}}=\frac{\|D u\|^2_{L^2(P^r_N)}}{\|u\|^2_{L^2(P^r_N)}}$, which together with $u\in\left\{v\in H^1(T),\, v\not\equiv0,\, v=0 \;\text{on $\gamma_1$} \right\}$ yields $\mu(T)\le\lambda(P^r_N)$. \\
On the other hand, 
every point $x$ in $P^r_N$ is image of a unique point $y$ of $T$ through some of the elements of the group of symmetries of $P_N^r$, namely $x=\mathcal{S}^i_k y$ for some $(i=1,2 \mbox{ and } k=1,...,N)$.
Then we set $\tilde w(x)=w(y)$. 
By construction $w=\tilde w$ on $T$ and $\frac{\|D \tilde w\|^2_{L^2(T)}}{\|\tilde w\|^2_{L^2(T)}}=\frac{\|D \tilde w\|^2_{L^2(P^r_N)}}{\|\tilde w\|^2_{L^2(P^r_N)}}$, therefore implying $\mu(T)\ge\lambda(P^r_N)$.
\end{proof}

Now that we have proved the equivalence between problem \eqref{eigenvalue_prob} on $P^r_N( r)$ and problem \eqref{mixed} on $T(\pi/N, r)$ we observe that $\mu(T(\alpha, r))$ is defined as a function of the parameter $\alpha$ for all $\alpha\in(0,\pi/2)$, opening the possibility, in what follows, to investigate the dependence of $\mu$ with respect to $\alpha$ by way of differentiation. 

From now on, when there is no confusion, we write for simplicity $\mu$ instead of $\mu(T(\alpha, r))$. Moreover we choose a reference frame in which (see Figure \ref{T}) the triangle $T(\alpha,r)$ lies in the first quadrant, the corner of the angle measuring $\alpha$ coincides with the origin $O$ and the cathetus $\gamma_3$ and $\gamma_1$ are parallel to $x$ and $y$ axis respectively.   
\begin{lemma}\label{lem_diff}

\begin{figure}
\def\svgwidth{.6\textwidth}
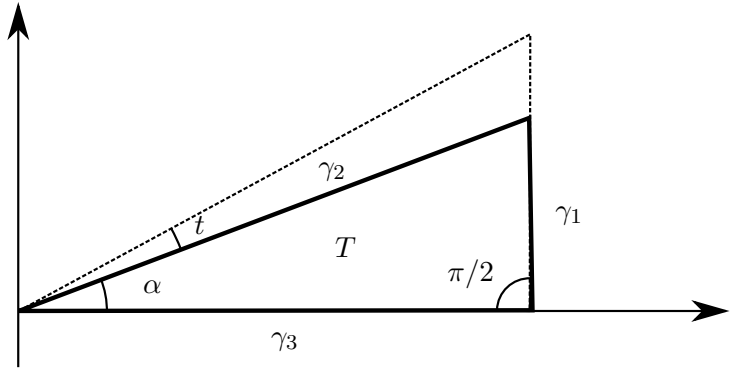\caption{The right triangle $T(\alpha,r)$ and its deformation with respect to $t$.}\label{T}
\end{figure}

For any given positive $r$, and for all $\alpha\in(0,\pi/2)$, if $v$ is a solution to \eqref{mixed} then we have
\[
\frac{\partial \mu}{\partial \alpha}=\frac{\displaystyle\int_{0}^{ r} \Big(|Dv(s \sin\alpha,s \cos\alpha)|^2-\mu v^2(s \sin\alpha,s \cos\alpha)\Big) s\ ds}{\displaystyle\int_{T} v^2}+2\mu\tan\alpha.
\]
\end{lemma} 
\begin{proof}
For all $t>0$ sufficiently small we can consider the triangle $T\left(\alpha+t,r\frac{\cos\alpha}{\cos(\alpha+t)}\right)$ (see Figure \ref{T} the dashed line). 
In such a way we define a one parameter family of domains which are perturbations of $T(\alpha,r)$. There exists in particular a smooth vector field $V$ (not uniquely defined everywhere) such that for all positive $t$ sufficiently small $T\left(\alpha+t,r\frac{\cos\alpha}{\cos(\alpha+t)}\right)=(\mathbb{I}+tV) T(\alpha,r)$, $\mathbb{I}$ being the identity. This is all what we need to use the Hadamard formula (see \cite{Hen}) to get
\begin{equation}\label{shapediff}
\left.\frac{d} {d t}\left[\mu\left(T\left(\alpha+t,r\frac{\cos\alpha}{\cos(\alpha+t)}\right)\right)\right]\right|_{t=0}=\frac{\displaystyle\int_{\gamma_2} \Big(|Dv|^2-\mu v^2\Big)V\cdot\nu\ d\sigma}{\displaystyle\int_{T} v^2}.
\end{equation}

Actually such a formula has been implemented to differentiate Neumann Laplacian eigenvalues with respect to domain variations. We are not dealing with a Neumann eigenvalue, nevertheless it is still possible to use the very same formula in our case since we are applying a deformation affecting only the Neumann part of the boundary. Moreover it is also easy to see that if we glue together $v$ and the reflection of $-v$ across the cathetus $\gamma_1$, we get a Neumann eigenfunction for $T\cup \tilde T$, where $\tilde T$ is the reflection of $T$ across $\gamma_1$. 

Then we observe that from \eqref{eq_scaling} we have $\frac{d}{dr}(r^2\mu(T(\alpha,r)))=0$ yielding
\[
\left.\frac{d} {d t}\left[\mu\left(T\left(\alpha+t,r\frac{\cos\alpha}{\cos(\alpha+t)}\right)\right)\right]\right|_{t=0}=\frac{\partial\mu}{\partial \alpha}-\frac{\partial\mu}{\partial r} r\tan\alpha=\frac{\partial\mu}{\partial \alpha}-2\mu\tan\alpha.
\]

Now we go back to the righthand side of \eqref{shapediff} and we parametrize the hypothenuse $\gamma_2$. We set $$x(s)=s\cos\alpha\quad \mbox{and}\quad y(s)=s\sin\alpha\quad \mbox{for $s\in[0,r]$}.$$ We have $d\sigma=ds$ and we observe that $\nu=(-\sin\alpha,\cos\alpha)$ on $\gamma_2$. 
Eventually we conclude the proof observing that, no matter which explicit representation of $V$ we choose, necessarily by construction we must have $V(x(s),y(s))=\left(0,\dfrac{s}{\cos\alpha}\right)$.
\end{proof}

\begin{lemma}\label{concavity}
Let $v$ be a solution to \eqref{mixed}, then we have for a.e. $x\in(0,r\cos\alpha)$
\[
v^2(x,x\tan\alpha)<\displaystyle\frac{1}{x\tan\alpha}\int_{0}^{ x\tan\alpha} v^2(x,y)\ dy.
\]
\end{lemma}
\begin{proof}
We can assume that $v$ is positive, if not we can consider $-v$. The Lemma is an immediate consequence of the fact that the function $h=\frac{\partial v}{\partial y}$ is everywhere negative inside $T$. To prove that the sign of $h$ is constant and non positive we observe that $-\Delta h=\mu h$ in $T$ and $h\le 0$ on $\partial T$. If by contradiction there exists an open set $D\subset T$ such that $h> 0$ on $D$ and $h=0$ on $\partial D$ then we would have $$\mu=\frac{\|Dh\|^2_{L^2(D)}}{\|h\|^2_{L^2(D)}}> \min\left\{\frac{\|D v\|^2_{L^2(T)}}{\|v\|^2_{L^2(T)}}:\,v\in H^1(T),\, v\not\equiv0,\, v=0 \;\text{on $\gamma_1$} \right\}=\mu,$$
which yields a contradiction. Once we know that $h$ is non positive in $T$ the by classical elliptic estimates it is strictly negative inside.
\end{proof}

\begin{lemma}\label{lem_main}
For any given positive $ r$ and for all $\alpha\in(0,\pi/2)$, we have
\[
\frac{\partial \mu}{\partial \alpha}\ge \mu\tan\alpha-\left(\mu-\frac{j_0^2}{ r^2 \cos^2\alpha}\right)\frac{1}{\tan\alpha}.
\]
\end{lemma}
\begin{proof}
Let $v$ be a solution to \eqref{mixed} and for $s\in[0, r]$ we set $g(s)=v(s \cos\alpha,s \sin\alpha)$. Since $\frac{\partial v}{\partial\nu}=0$ on $\gamma_2$ then $|Dv(s \cos\alpha,s \sin\alpha)|=-g'(s)$. As we said before $\displaystyle\frac{j_0^2}{r^2}$ is first Dirichlet Laplacian eigenvalue on the disk $B_r$ of radius $r$. This means that 
\[
\frac{j_0^2}{r^2}=\min\left\{\frac{\|D u\|^2_{L^2(B_r)}}{\|u\|^2_{L^2(B_r)}}:\,u\in H^1_0(B_r),\, u\not\equiv0 \right\}.
\]
Since first eigenfunctions are radial (they have the same symmetry as the domain $B_r$), then the minimum can be taken on radial functions
\begin{equation}\label{eigenball}
\frac{j_0^2}{r^2}=\min\left\{\frac{\int_{0}^{ r} |u'(s)|^2 s\ ds}{\int_{0}^{ r} u^2(s) s\ ds}:\,u\in H^1(0,r),\, u\not\equiv0, \, u(r)=0 \right\}.
\end{equation}

Moreover, using again the radial symmetry,
\begin{equation}\label{eigensector}
\frac{j_0^2}{r^2}=\min\left\{\frac{\|D u\|^2_{L^2(C_{\alpha r})}}{\|u\|^2_{L^2(C_{\alpha r})}}:\,u\in H^1(C_{\alpha r}),\, u\not\equiv0, u(x,y)=0 \mbox{ if $x^2+y^2=r^2$} \right\},
\end{equation}
where $C_{\alpha r}$ is a circular sector of radius $r$ and opening angle $\alpha$, center in the origin and containing the triangle $T(\alpha,r)$. By standard arguments any solution to \eqref{mixed} (extended to zero outside $T$) is an admissible function to be used in \eqref{eigensector}, and we have $\displaystyle\frac{j_0^2}{r^2}<\mu$.

By Lemma \ref{lem_diff}, the characterization \eqref{eigenball} and Lemma \ref{concavity}, we have
\begin{gather*}
\begin{split}
\frac{\partial \mu}{\partial \alpha}&=2\mu\tan\alpha+\frac{\displaystyle\int_{0}^{ r} \Big(|g'(s)|^2-\mu g^2(s)\Big) s\ ds}{\displaystyle\int_{T} v^2}\\
&\ge2\mu\tan\alpha-\left(\mu-\frac{j_0}{ r^2}\right)\frac{\displaystyle\int_{0}^{ r} g^2(s) s\ ds}{\displaystyle\int_{T} v^2}\\
&=2\mu\tan\alpha-\left(\mu-\frac{j_0}{ r^2}\right)\frac{\displaystyle\int_{0}^{ r\cos\alpha} v^2(x,x\tan\alpha) \frac{x}{\cos^2\alpha}\ dx}{\displaystyle\int_{T} v^2}\\
&\ge 2\mu\tan\alpha-\left(\mu-\frac{j_0}{ r^2}\right)\frac{\displaystyle\int_{0}^{ r\cos\alpha} \left(\int_{0}^{ r\sin\alpha} v^2(x,y)\frac{1}{\sin\alpha\cos\alpha}\ dy\right)dx }{\displaystyle\int_{T} v^2}\\
&= 2\mu\tan\alpha- \left(\mu-\frac{j_0}{ r^2}\right)\frac{1}{\sin\alpha\ \cos\alpha}
\end{split}
\end{gather*}
and the proof is completed.
\end{proof}
 
Now, since $\frac{j_0^2}{ r^2\cos^2\alpha}=\lambda(B_{ r\cos\alpha})$, and $B_{ r\cos\alpha}\subset P^ r_N$ then
\[
\frac{j_0^2}{ r^2\cos^2\alpha}> \mu.
\] 
Therefore by Lemma \ref{lem_main}
\[
\frac{\partial \mu}{\partial \alpha}> \mu\tan\alpha
\]
and integrating in $\alpha$ from $\frac{\pi}{N+1}$ to $\frac{\pi}{N}$ we get Theorem \ref{teo_ref}.

At this point we can prove Theorem \ref{teo_best}. Using Lemma \ref{lem_main} in fact we have
\[
\frac {\partial}{\partial\alpha}\Big( r^2 \sin\alpha\cos\alpha \,\mu(T(\alpha, r))\Big)\ge j_0^2,
\]
which, integrated with respect to $\alpha$, from $\frac{\pi}{N+1}$ to $\frac{\pi}{N}$ gives
\[
 r^2 \sin \frac{\pi}{N}\cos\frac{\pi}{N} \,\lambda(P^ r_N)  -  r^2 \sin \frac{\pi}{N+1}\cos\frac{\pi}{N+1} \,\lambda(P^ r_{N+1} )  
\ge \frac{\pi j_0^2}{N(N+1)}.
\]
\vspace{0.4in}

\centerline{\bf Acknowledgements}
This work has been partially supported by Progetto Gnampa 2013 ``Disuguaglianze funzionali e problemi sovradeterminati'' and was completed when the Author was Humboldt Fellow at the Mathematisches Institut of the University of Cologne.

\end{document}